\documentclass[12pt]{amsart}
\usepackage{graphicx}
\usepackage{delimset, amsrefs, xcolor}
\usepackage[colorlinks, citecolor=cyan]{hyperref}

\usepackage{enumitem}
\setlist[enumerate,itemize]{itemsep=3pt}
\usepackage[capitalize]{cleveref}
\usepackage{amsmath}
\newtheorem{thm}{Theorem}[section]

\newtheorem{cor}[thm]{Corollary}
\newtheorem{lem}[thm]{Lemma}
\newtheorem{prop}[thm]{Proposition}

\newtheorem{rem}[thm]{Remark}

\theoremstyle{remark}

\DeclareMathOperator{\des}{\mathrm{des}}
\DeclareMathOperator{\asc}{\mathrm{asc}}
\DeclareMathOperator{\cdes}{\mathrm{cdes}}
\DeclareMathOperator{\casc}{\mathrm{casc}}

\DeclareMathOperator{\D}{\mathrm{D}}
\newcommand\floor[1]{\lfloor#1\rfloor}
\usepackage{mathrsfs} 
\newcommand{\B}{\mathscr{B}}

\renewcommand{\O}{\mathscr{O}}

\newcommand{\Eulerian}[2]{\genfrac{<}{>}{0pt}{1}{#1}{#2}}

\crefname{thm}{Theorem}{Theorems}
\creflabelformat{thm}{~\upshape #2#1#3}

\numberwithin{equation}{section}
\numberwithin{figure}{section}
\numberwithin{table}{section}

\allowbreak
\allowdisplaybreaks

\title{Refined Eulerian numbers and ballot permutations}
\author[Tongyuan Zhao]{Tongyuan Zhao$^1$}
\address{
$^2$College of Sciences,
China University of Petroleum,
102249 Beijing, P. R. China}
\email{zhaotongyuan@cup.edu.cn}

\author[Y. Sun]{Yue Sun$^2$}
\address{
$^1$College of Sciences,
China University of Petroleum,
102249 Beijing, P. R. China}
\email{sunyuecup@sina.com}

\author[Feng Zhao]{Feng Zhao$^3$}
\address{
$^2$College of Mathematics and Information Science,
Hebei Normal University,
050024 Shijiazhuang, P.R. China}
\email{zhaofeng@hebtu.edu.cn}

\thanks{Zhao is supported
by Science Foundation of China University of Petroleum, Beijing
(Grant No.\ 2462020YXZZ004).}



\begin{document}
\maketitle
\begin{abstract}
A ballot permutation is a permutation $\pi$ such that in any prefix of $\pi$ the descent number is not more than the ascent number. In this article,
 we obtained a formula in close form for the multivariate generating function of $\{A(n,d,j)\}_{n,d,j}$, which denote the number of permutations of length $n$ with $d$ descents and $j$ as the first letter. Besides, by a series of calculations with generatingfunctionology, we confirm a recent conjecture of Wang and Zhang for ballot permutations.
\end{abstract}

\section{Introduction}
Let $\mathcal S_{n}$ be the symmetric group of all permutations of $[n]=\{1, 2, \dots , n\}$.
A position $i\in[n-1]$ in a permutation $\pi=\pi_1\pi_2\dotsm\pi_n\in\mathcal S_n$ is a \emph{descent} if $\pi_i>\pi_{i+1}$,
and an \emph{ascent} if $\pi_i<\pi_{i+1}$.
Denote the number of descents of $\pi$ by $\des(\pi)$,
and the number of ascents by $\asc(\pi)$.
We call the number $h(\pi)=\asc(\pi)-\des(\pi)$
the \emph{height} of $\pi$.
The permutation $\pi$ is said to be a \emph{ballot permutation}
if $h(\pi_1\pi_2\cdots\pi_i)\ge 0$
for all $i\in[n]$.
Let $\B_n$ denote the set of ballot permutations on $[n]$. Define $\B_{0}=\{\epsilon\}$,
where $\epsilon$ is the empty permutation.

The problem of enumerating ballot permutations is closely related with
that of enumerating ordinary permutations with
a given up–down signature, see \cite{And81,BFW07,Niv68,She12}.
A ballot permutation of length $2n+1$ with $n$ descents is said to be a \emph{Dyck permutation},
whose enumeration is the Eulerian-Catalan number, see Bidkhori and Sullivant~\cite{BS11}.
Let $\O_n$ be the set of \emph{odd order permutations}
of $[n]$, viz., the set of permutations of $[n]$ which are the products of cycles with odd lengths. Bernardi, Duplantier and Nadeau \cite{BDN10} proved
that the number of ballot permutations of length $n$
equals the number of odd order permutations of length $n$,
by using compositions of bijections, and thus \cref{thm:BDN} follows.
A short proof is given by Wang and Zhang~\cite{WZ20}.

\begin{thm}[Bernardi, Duplantier and Nadeau]\label{thm:BDN}
The number of ballot permutations of length $n$ is
\[
b_n=\begin{cases}
[(n-1)!!]^{2},&\text{if $n$ is even},\\
n!!(n-2)!!,&\text{if $n$ is odd},
\end{cases}
\]
where $(-1)!!=1$.
\end{thm}
In order to give a refinement for \cref{thm:BDN}, Spiro~\cite{Spi20} introduced a statistic $M(\pi)$ which is defined for a permutation $\pi$ such that
\[
M(\pi)=\sum_{c}\min\brk1{\cdes(c),\,\casc(c)},
\]
where the sum runs over all cycles $c=(c_1c_2\dotsm c_k)$ of $\pi$, with the \emph{cyclic descent}
\[
\cdes(c)=\abs{\{i\in[k]\colon c_i>c_{i+1}\ \text{where $c_{k+1}=c_1$}\}},
\]
and the \emph{cyclic ascent}
\[
\casc(c)=\abs{\{i\in[k]\colon c_i<c_{i+1}\ \text{where $c_{k+1}=c_1$}\}}
=\abs{c}-\cdes(c),
\]
where $\abs{c}$ is the length of $c$. Spiro conjectured that
the number of ballot permutations of length $n$ with $d$ descents
equals the number of odd order permutations $\pi$ of length $n$
such that $M(\pi)=d$, which is confirmed by Wang and the first author~\cite{WZ20a},
and thus \cref{thm:Spiro} follows.

\begin{thm}\label{thm:Spiro}
Let $n\ge 1$ and $0\le d\le \floor{(n-1)/2}$.
The number of ballot permutations of length $n$
with $d$ descents equals the number of odd order permutations $\pi$
of length $n$ with $M(\pi)=d$.
\end{thm}

\cref{thm:Spiro} is proved by computing their bivariate generating functions in terms of the \emph{Eulerian number}~\cite{WZ20}.
For $n\ge1$ and $0\le d\le n-1$,
the Eulerian number, denoted as $A(n,d)$ or $\Eulerian{n}{d}$,
is the number of
permutations of $[n]$ with $d$ descents, see OEIS~\cite[A008292]{OEIS}.
We adopt the convention $A(0,0)=1$ and
\[
A(n,d)=0,\qquad\text{if $n<0$, or $d<0$, or $d=n\ge 1$, or $d>n$}.
\]
The $n$th \emph{Eulerian polynomial} is
\[
A_n(t)=\sum_{\pi\in\mathcal S_n}t^{\des(\pi)}
=\sum_{d}A(n,d)t^{d}\quad\text{for $n\ge 1$},
\]
where the notation $\sum_{i}$
implies that the index $i$ runs over all nonnegative integers
making the summation meaningful, and $A_0(t)=1$,
see Kyle Petersen~\cite[\S 1.4]{Kyle15B}.
The exponential generating function of the Eulerian polynomials is
\begin{equation}\label{gf:A}
\sum_{n}A_{n}(t)\frac{x^{n}}{n!}
=\frac{t-1}{t-e^{(t-1)x}},
\end{equation}
see \cite[Theorem 1.6]{Kyle15B} and \cite[Formula (75)]{FS09B}. Then the bivariate generating function
\begin{equation}\label{Eulequ}
A(t,x)
=\sum_{n\ge1}\sum_{d}\frac{A(n,d)t^{d}x^{n}}{n!}
=\sum_{n\ge1}A_n(t)\frac{x^{n}}{n!}\\
=\frac{t-1}{t-e^{(t-1)x}}-1
=\frac{e^{(1-t)x}-1}{1-t e^{(1-t)x}}.
\end{equation}
Denote by $b(n,d)$ the number of ballot permutations of length $n$ with $d$ descents,
and denote the bivariate generating function of $\{b(n,d)\}_{n,d}$ by
\[
B(t,x)=\sum_{n}\sum_{d\le n-1}
\frac{b(n,d)\,t^{d}x^{n}}{n!}.
\]
Wang and the first author~\cite{WZ20a} give the following Theorem by calculating the joint distribution of the peak and descent statistics over ballot permutations.
\begin{thm}[Wang and Zhao]\label{Wang and Zhao}
We have
\[
B(t,x)=\exp\brk3{x+2 \sum_{k\ge 1}\sum_{d\le k-1}
A(2k,d)t^{d+1}\frac{x^{2k+1}}{(2k+1)!}}.
\]
\end{thm}

Wang and Zhang~\cite{WZ20}
gave another conjecture which refined \cref{thm:Spiro} by tracking the neighbors of the
largest letter in these permutations. They defined a word $u$ as a \emph{factor} of a word $w$ if there exist words $x$ and $y$ such that
$w = xuy$, and a word $u$ as a \emph{cyclic factor} of a permutation $\pi\in\mathcal{S}_n$ if $u$ is a factor of some word $v$ such that $(v)$ is a cycle of $\pi$. This conjecture is confirmed in this paper by computing their multivariate generating functions respectively,
and thus \cref{thm:WZ} follows.

\begin{thm}[Wang and Zhang]\label{thm:WZ}
For all $n$, $d$, and $2\leq j \leq n-1$, we have $b_{n,d}(1,j)+b_{n,d}(j,1)= 2p_{n,d}(1,j)$,
where $b_{n,d}(i,j)$ is the number of ballot permutations of length $n$ with $d$ descents which have $inj$ as a factor, and $p_{n,d}(i,j)$ is the number of odd order permutations of length $n$ with $M(\pi)=d$ which have $inj$ as a cyclic factor.
\end{thm}

With \cref{thm:WZ} and the Toeplitz property of $b_{n,d}(i,j)+b_{n,d}(j,i)$ and $p_{n,d}(i,j)$ showed in ~\cite{WZ20}, we obtain in \cref{RS} the generating function which P. R. Stanley expressed interest in private communication, and is defined by
\[
P(t,x,y,z)=\sum_{n,d}\sum_{1\leq i<j\leq n-1}\dfrac{2p_{n,d}(i,j)t^{d}x^{n}y^{i}z^{j}}{(j-i-1)!(n-j+i-2)!}.
\]

There are various generalizations of Eulerian numbers. Here we focus on a refinement $A(n,d,j)$ of $A(n,d)$ defined by Brenti and Welker~\cite{BW08}, which is the number of permutations of length $n$ with $d$ descents and first letter $j$ for $n\geq 1$, $0\leq d\leq n-1$ and $1\leq j\leq n$. They show the real-rootedness of the following polynomials
\[
A^{\langle j\rangle}_n(t)=\sum_{d}A(n,d,j)t^{d}\quad\text{for $n\ge 1$},
\]
which refine the Eulerian polynomials. We adopt the convention
$A(n,d,j)=0$ for other $n,d$ and $j$. Define
\[
A(t,x,y):=\sum_{n,d}\sum_{j=1}^{n}\frac{A(n,d,j)t^{d}x^{n}y^{j}}{(j-1)!(n-j)!},
\]
We would give a formula for $A(t,x,y)$ in close form, based what a series of formulas with generatingfunctionology could be established, and prove \cref{thm:WZ}.

For this purpose, we give several definitions of enumerating sequences
and calculate their multivariate generating functions in the following sections.
For $n\geq2, 1\leq j\leq n$ and
$0\leq d\leq n$, let $U(n,\,d,\,j)$ denote the number of
permutations of length $n$ with $d-1$ descents or $n-d-1$ ascents and
first letter $j$. Then
\begin{equation}\label{undj}
U(n,d,j)=A(n,\, n-d-1,\, j)+A(n,\, d-1,\, j),
\end{equation}
and
\begin{equation}\label{undj1}
U(n,d,j)=U(n,\, n-d,\, j).
\end{equation}
We define the multivariate generating function for $\{U(n,d,j)\}_{n,d,j}$
\[
\hat{U}(t,x,y):=\sum_{n,d}\sum_{j=1}^{n}\frac{U(n,d,j)t^{d}x^{n}y^{j}}{(j-1)!(n-j)!},
\]
and
\[
U(t,x,y):=\sum_{n ~\text{is odd},j}\sum_{d\leq\frac{n-1}{2}}\frac{U(n,d,j)t^{d}x^{n}y^{j}}{(j-1)!(n-j)!}.
\]
For $n\geq 3, 2\leq j\leq n-1$ and $0\leq d\leq n-1$, we denote by $E(n,d,j)$ the number of permutations in $S_n$ of length $n$ with $d$ descents which have $1nj$ or $jn1$ as a factor,
and define $b(n,d,j):=b_{n,d}(1,j)+b_{n,d}(j,1)$.
Define
\[
E(t,x,y):=\sum_{n,d}\sum_{j=2}^{n-1}\dfrac{E(n,d,j)t^{d}x^{n}y^{j}}{(n-j-1)!(j-2)!},
\]
\[
B(t,x,y):=\sum_{n,d}\sum_{j=2}^{n-1}\dfrac{b(n,d,j)t^dx^ny^j}{(n-j-1)!(j-2)!},
\]
and
\[
P(t,x,y):=\sum_{n,d}\sum_{j=2}^{n-1}\dfrac{p_{n,d}(1,j)t^{d}x^{n}y^{j}}{(j-2)!(n-j-1)!}.
\]

The rest of this paper is organized as follows. In Section 2, we calculate the formulas for $A(t,x,y), \hat{U}(t,x,y)$ and $U(t,x,y)$. In Section 3, we calculate the formulas for $E(t,x,y)$, $P(t,x,y)$ and a relation between $B(t,x), B(t,x,y)$ and $E(t,x,y)$. Finally,
in Section 4 we confirm the conjecture of Wang and Zhang and calculate $P(t,x,y,z)$.

\section{Formulas for $A(t,x,y), \hat{U}(t,x,y)$ and $U(t,x,y)$}
This section is devoted to giving the generating functions of
$\{A(n,d,j)\}_{n,d,j}$, which is the basis of section 3.
One can get the following recursion for $A(n,d,j)$ by the values of the second letters in permutations, which could be found in \cite{BW08}.
\begin{lem} \label{BWlem}
For $n\geq 2$, $1\leq j\leq n$ and $0\leq d\leq n-1$, we have
		\begin{equation}
		\begin{split}\label{ditui1}
		A(n,d,j)=\sum_{i=1}^{j-1}A(n-1,\,d-1,\,i) + \sum_{i=j+1}^{n} A(n-1,\,d,\,i-1).
		\end{split}
		\end{equation}
\end{lem}

By \cref{BWlem}, we can obtain the following Theorem by generating function calculation.
\begin{thm}\label{yinliv2}
We have
	$$A(t,x,y)=\dfrac{(t-1)xye^{(t-1)xy}}{t-e^{(t-1)x(1+y)}}.$$
\end{thm}
\begin{proof}
Substituting $j$ by $j+1$ in \eqref{ditui1},we have
		\begin{equation}
		\begin{split}\label{ditui2}
		A(n,\,d,\,j+1)=\sum_{i=1}^{j}A(n-1,\,d-1,\,i)+\sum_{i=j+2}^{n}A(n-1,\,d,\,i-1).	
		\end{split}
		\end{equation}
subtracting \eqref{ditui1} from \eqref{ditui2},we get
	\begin{equation}
	\begin{split}\label{zuocha}
	A(n,\,d,\,j+1)=A(n,d,j)+A(n-1,\,d-1,\,j)-A(n-1,\,d,\,j),
	\end{split}
	\end{equation}
where $n\geq 2$, $1\leq j\leq n-1$ and $0\leq d\leq n-1$.
	
Multiplying each term in \eqref{zuocha} by$ \dfrac{t^{d}x^{n}y^{j}}{(j-1)!(n-j)!}$ and summing over all integers $n,d$ and $1 \leq j \leq n-1$ with subscript transformation, we have
\begin{align}\label{key2}
\sum_{n,d}\sum_{1\leq\,j\leq\,n-1}\dfrac{A(n,\,d,\,j+1)t^{d}x^{n}y^{j+1}}{(n-j-1)!(j-1)!}
&=\,\sum_{n,d}\sum_{j\geq\,1}\dfrac{(j-1)A(n,\,d,\,j)t^{d}x^{n}y^{j}}{(n-j)!(j-1)!}\notag\\
&=y\sum_{n,d}\sum_{j\geq1}\dfrac{A(n,d,j)jt^{d}x^{n}y^{j-1}}{(n-j)!(j-1)!}-A(t,x,y)\notag\\
&=y\dfrac{\partial A(t,x,y)}{\partial y}-A(t,x,y),
\end{align}
\begin{align}\label{key3}
\sum_{n,d}\sum_{1\leq\,j\leq\,n-1}\dfrac{A(n,d,j)t^{d}x^{n}y^{j+1}}{(n-j-1)!(j-1)!}
&=\,\sum_{n,d}\sum_{1\leq\,j\leq\,n-1}A(n,d,j)\dfrac{(n-j)t^{d}x^{n}z^{j+1}}{(n-j)!(j-1)!}\notag\\
&=xy\sum_{n,d}\sum_{1\leq\,j\leq\,n-1}\dfrac{nA(n,d,j)t^{d}x^{n-1}y^{j}}{(n-j)!(j-1)!} \notag\\ &\quad-y^{2}\sum_{n,d}\sum_{1\leq\,j\leq\,n-1}\dfrac{jA(n,d,j)t^{d}x^{n}y^{j-1}}{(n-j)!(j-1)!}\notag\\
&=xy\dfrac{\partial A(t,x,y)}{\partial x}-y^{2}\dfrac{\partial A(t,x,y)}{\partial y},
\end{align}	
\begin{align}\label{key4}
\sum_{n,d}\sum_{1\leq\,j\leq\,n-1}\dfrac{A(n-1,\,d-1,\,j)t^{d}x^{n}y^{j+1}}{(n-j-1)!(j-1)!}
=txy A(t,x,y),
\end{align}
\begin{align}\label{key5}
\sum_{n,d}\sum_{1\leq\,j\leq\,n-1}\dfrac{A(n-1,\,d,\,j)t^{d}x^{n}y^{j+1}}{(n-j-1)!(j-1)!}
=xy A(t,x,y).
\end{align}
Combining \eqref{key2}, \eqref{key3}, \eqref{key4} and \eqref{key5}, we get
$$y\dfrac{\partial A(t,x,y)}{\partial y}-A(t,x,y)=xy\dfrac{\partial A(t,x,y)}{\partial x}-y^{2}\dfrac{\partial A(t,x,y)}{\partial y}+txy A(t,x,y)-xy A(t,x,y).$$
By Maple, we have
	\begin{equation}\label{v2jielun}
		\begin{split}
			A(t,x,y)=xye^{(t-1)xy}F(t,x+xy),
		\end{split}
	\end{equation}
where $F(t,x)$ is an arbitrary differentiable function.
	
In the following we will determine $F(t,x)$.
Take derivatives on both sides of \eqref{v2jielun}, we have
\[
\dfrac{\partial A(t,x,y)}{\partial y}=x[(t-1)xy+1]e^{(t-1)xy}F(t,x+xy)+x^{2}ye^{(t-1)xy}F_{2}^{'}(t,x+xy),
\]
then
\[
\dfrac{\partial A}{\partial z}(t,x,0)=xF(t,x).	
\]
On the other hand, it is easy to see that $A(n,d,1)=A(n-1,\,d)$, then
\[
\dfrac{\partial A}{\partial z}(t,x,0)=\sum_{n,d} \dfrac{A(n,d,1)t^{d}x^{n}}{(n-1)!}
=x\,\sum_{n,d}\dfrac{A(n,d)\,x^{n}y^{d}}{n!} =xA(t,x)+x.
\]
Thus
\begin{equation}\label{F}
F(t,x)=A(t,x)+1=\frac{t-1}{t-e^{(t-1)x}}.
\end{equation}
Plugging \eqref{F} into \eqref{v2jielun},one can complete the proof.
\end{proof}
\begin{rem}
The close form of the generating function for $\{A(n,d,j)\}$ in \eqref{yinliv2}
doesn't seem to be seen elsewhere. The conclusion could also be obtained by
generating function calculation through Zhang's lemma in \cite{Zhang20}: for $1\leq j\leq n$,
\[
\frac{A^{\langle l\rangle}_n(x)}{(1-x)^n}=\delta_{l,1}+\sum_{i=1}^{\infty}j^{l-1}(j+1)^{n-l}x^j.
\]
\end{rem}

It is direct to calculate $\hat{U}(t,x,y)$ according to the definition
of \eqref{undj} and \cref{yinliv2}.
\begin{cor} \label{utxy2}
We have
\[
\hat{U}(t,x,y)=tA(t,x,y)+\frac{1}{t}A(\frac{1}{t},\,xt,\,y),
\]
i.e.,
\[
\hat{U}(t,x,y)=\frac{(t-1)xy(te^{(t-1)xy}+e^{(t-1)x})}{t-e^{(t-1)x(y+1)}}.
\]
\end{cor}

With \cref{utxy2}, we obtain the following corollary.
\begin{prop} \label{utxy3} we have
\[
U(t,x,y)+U(\frac{1}{t},tx,y)=\dfrac{\hat{U}(t,x,y)-\hat{U}(t,-x,y)}{2}.
\]
In other words,
\[
U(t,x,y)=\D^{t,x}\brk4{\dfrac{xy(t+1)(t-1)^2e^{(t-1)x}\brk1{e^{2(t-1)xy}+1}}
{2\brk1{e^{(t-1)x(y+1)}}-1\brk1{t-e^{(t-1)x(y+1)}}}}.
\]
\end{prop}
\begin{proof}
According to \eqref{undj1} and using subscript transformation, we have
\begin{align*}
U(t,x,y)+U(\frac{1}{t},tx,y)
&=\sum_{n ~\text{is odd},j}\sum_{d\leq\frac{n-1}{2}}\frac{U(n,d,j)t^{d}x^{n}y^{j}}{(j-1)!(n-j)!}
+\sum_{n ~\text{is odd},j}\sum_{d\leq\frac{n-1}{2}}\frac{U(n,d,j)t^{n-d}x^{n}y^{j}}{(j-1)!(n-j)!}\\
&=\sum_{n ~\text{is odd},d,j}\frac{U(n,d,j)t^{d}x^{n}y^{j}}{(j-1)!(n-j)!}=\dfrac{\hat{U}(t,x,y)-\hat{U}(t,-x,y)}{2}\\
&=\dfrac{xy(t+1)(t-1)^2e^{(t-1)x}\brk1{e^{2(t-1)xy}+1}}
{2\brk1{e^{(t-1)x(y+1)}}-1\brk1{t-e^{(t-1)x(y+1)}}}.
\end{align*}
Noting that
$U(t,x,y)$ is a multivariate formal power series
with terms of the form $t^{d}x^{n}y^{k}$ such that $d \leq (n-1)/2$,
and the terms of $U(\frac{1}{t},tx,y)$ are of the form
$t^{d}x^{n}y^{k}$ such that $d >(n-1)/2$.
According the definition of $\D^{t,x}$, we obtain the desired equation.
\end{proof}

\section{Formulas for $E(t,x,y)$, $P(t,x,y)$ and $B(t,x,y)$}
\subsection{The generating function of $\{E(n,d,j)\}_{n,d,j}$}
 \begin{thm} \label{recEndj}
 For $n\geq3$, $2\leq j \leq n-1$ and $0\leq d\leq n-1$, We have
 \begin{multline}\label{Endj}
  E(n,d,j) =\sum_{l=1}^{n-2}\sum_{k=0}^{l-1}\sum_{u=0}^{j-2}\binom{j-2}{u}\binom{n-j-1}{l-1-u}A(n-l-2, d-k-1)U(l,k,u+1)\\+V(n-2, d-1, j-1)-V(n-2, d-2, j-1),
\end{multline}
where $U(l,k,u+1)=V(l, l-1-k, u+1)+V(l, k-1, u+1)$.
 \end{thm}
 \begin{proof}
  First we give a formula of $E(n,d,j)$ with which one can derive the
  expression of $E(t,x,y)$.

  For any permutation $\pi=\pi_1 \pi_2 \cdots \pi_n$ of length $n$ with $d$ descents which has $1nj$ or $jn1$ as a factor, there are three cases of $\pi$:
 \begin{itemize}
 \item
 $jn1$ is in $\pi$. Suppose that $\pi_l=j$, $\des(\pi_1 \pi_2 \cdots \pi_j)=k$
 and there are $u$ letters in $\{\pi_1, \pi_2,\dotsm \pi_{l-1}\}$ which
 are less than $j$, where $1\leq k+1 \leq l\leq n-2$ and $0\leq u \leq j-2$.
 Then $\pi_{l}\pi_{l-1}\dotsm\pi_1$ is a permutation of length $l$
 with the $(u+1)$th largest letter as the first letter, and
 \[
 \des(\pi_{l}\pi_{l-1}\dotsm\pi_1)=l-1-\asc(\pi_{l}\pi_{l-1}\dotsm\pi_1)=l-1-k.
 \]
 Considering $n$ is a descent and 1 is a ascent in $\pi$,
 $\pi_{l+3}\pi_{l+4}\dotsm\,\pi_n$ is a permutation of length $n-l-2$ with
 \[
 \des(\pi_{l+3}\pi_{l+4}\dotsm\,\pi_n)=d-1-\des(\pi_1 \pi_2 \cdots \pi_j)=d-k-1.
 \]
 Note that as a subset of $[n]$, there are $\binom{j-2}{u}\binom{n-j-1}{l-1-u}$
 possibilities of $\{\pi_1, \pi_2,\dotsm \pi_{l-1}\}$. Thus the number of
 such $\pi$s is
\[
\sum_{l=1}^{n-2}\sum_{k=0}^{l-1}\sum_{u=0}^{j-2}\binom{j-2}{u}\binom{n-j-1}{l-1-u}A(l, l-1-k, u+1)A(n-l-2, d-k-1).
\]
 \item
 $1nj$ is in $\pi$ and $\pi_1\neq1$. Suppose that
 $\pi_{n-l+1}=j$, $\des(\pi_{n-l+1} \pi_{n-l+2}\dotsm \pi_{l-1})=k-1$
 and there are $u$ letters in $\{\pi_{n-l+1}, \pi_{n-l+2},\dotsm \pi_{l-1}\}$
 which are less than $j$, where $1\leq k+1 \leq l\leq n-3$ and $0\leq u \leq j-2$.
 Similar to case 1, $\pi=\pi_1\pi_2\dotsm\pi_{n-l-2}$ is a permutation of length $n-l-2$ with $d-1-k$ descents and $j\pi_{n-l+2}\pi_{n-l+3}\dotsm\pi_n$ is a permutation of length $l$
 with $k-1$ descents and the $(u+1)$th largest letter as the first letter.
 Thus the number of  such $\pi$s is
\[
\sum_{l=1}^{n-3}\sum_{k=0}^{l-1}\sum_{u=0}^{j-2}\binom{j-2}{u}\binom{n-j-1}{l-1-u}V(l, k-1, u+1)A(n-l-2, d-k-1).
\]
\item
$1nj$ is in $\pi$ and $\pi_1=1$. So the $\pi_3=j$. Then $j\pi_4\pi_5\dotsm\pi_n$ is a permutation of length $n-2$ with $d-1$ descents and the $(j-1)$th largest letter as the first letter. The number of such $\pi$s is $V(n-2, d-1, j-1)$.

\end{itemize}

Combining the above three cases, we have
\begin{multline*}		E(n,d,j)=\sum_{l=1}^{n-2}\sum_{k=0}^{l-1}\sum_{u=0}^{j-2}\binom{j-2}{u}\binom{n-j-1}{l-1-u}V(l, l-1-k, u+1)A(n-l-2, d-k-1)\\
+\sum_{l=1}^{n-3}\sum_{k=0}^{l-1}\sum_{u=0}^{j-2}\binom{j-2}{u}\binom{n-j-1}{l-1-u}V(l, k-1, u+1)A(n-l-2, d-k-1)+V(n-2, d-1, j-1)\\
=\sum_{l=1}^{n-2}\sum_{k=0}^{l-1}\sum_{u=0}^{j-2}\binom{j-2}{u}\binom{n-j-1}{l-1-u}A(n-l-2, d-k-1)U(l,k,u+1)\\
-\sum_{k=0}^{n-2}\sum_{u=0}^{j-2}\binom{j-2}{u}\binom{n-j-1}{l-1-u}V(l, k-1, u+1)A(n-l-2, d-k-1)
+V(n-2, d-1, j-1).
\end{multline*}
It is easy to check that when $l=n-2$,
\[
\sum_{k=0}^{l}\sum_{u=0}^{j-2}\binom{j-2}{u}\binom{n-j-1}{l-1-u}V(l, k-1, u+1)A(n-l-2, d-k-1)= V(n-2, d-2, j-1),
\]
which complete the proof.
\end{proof}

\cref{recEndj} is translated into the language of generating functions as follows.
\begin{thm}\label{D2}
\[
E(t,x,y)=tx^2y[A(t,x+xy)+1]\hat{U}(t,x,y)+t(1-t)x^2yA(t,x,y),
\]
i.e.,
\[
E(t,x,y)=\dfrac{t(t-1)^2x^3y^2e^{(t-1)x}(e^{2(t-1)xy}+1)}
{(t-e^{(t-1)x(y+1)})^2}.
\]
\end{thm}
\begin{proof}
Multiplying each term in \eqref{Endj} by  $\dfrac{t^dx^ny^j}{(j-2)!(n-j-1)!}$,
summing over all integers $n,j$ and $d$ such that $n\geq3$, $2\leq j\leq n-1$
and $0\leq d\leq n-1$,
we deduce that
\begin{align*}
 E(t,x,y)
 &= \sum_{n=3}^{\infty}\sum_{d=0}^{n-1}\sum_{j=2}^{n-1}\frac{E(n,d,j)t^dx^ny^j}{(j-2)!(n-j-1)!} \\
 &= \sum_{n=3}^{\infty}\sum_{d=0}^{n-1}\sum_{j=2}^{n-1}\sum_{l=1}^{n-2}\sum_{k=0}^{l-1}
 \sum_{u=0}^{j-2}\frac{\binom{j-2}{u}\binom{n-j-1}{l-1-u})t^dx^ny^j}{(j-2)!(n-j-1)!}A(n-l-2, d-k-1)U(l,k,u+1)\\
 &\quad + tx^2y\sum_{n=3}^{\infty}\sum_{d=0}^{n-1}\sum_{j=2}^{n-1}\frac{A(n-2, d-1,j-1)t^{d-1}x^{n-2}y^{j-1}}{(j-2)!(n-j-1)!} \\
 &\quad - t^2x^2y\sum_{n=3}^{\infty}\sum_{d=0}^{n-1}\sum_{j=2}^{n-1}\frac{A(n-2, d-2, j-1)t^{d-2}x^{n-2}y^{j-1}}{(j-2)!(n-j-1)!}\\
 &=tx^2\sum_{n=3}^{\infty}\sum_{d=0}^{n-1}\sum_{j=2}^{n-1}\sum_{l=1}^{n-2}\sum_{k=0}^{l-1}
 \sum_{u=0}^{j-2}\binom{n-l-2}{j-2-u}y^j \cdot \frac{A(n-l-2, d-k-1)x^{n-l-2}t^{d-k-1}}{(n-2-l)!}\\
 &\quad \cdot\frac{U(l,k,u+1)t^kx^l}{u!(l-1-u)!}+tx^2yA(t,x,y)-t^2x^2yA(t,x,y).
\end{align*}
Exchanging the order of $u$ and $j$ in the summation, we have
\begin{multline*}
\sum_{n=3}^{\infty}\sum_{d=0}^{n-1}\sum_{j=2}^{n-1}\sum_{l=1}^{n-2}\sum_{k=0}^{l-1}
 \sum_{u=0}^{j-2}\binom{n-l-2}{j-2-u}y^j \cdot \frac{A(n-l-2, d-k-1)x^{n-l-2}t^{d-k-1}}{(n-2-l)!}\\
  \cdot\frac{U(l,k,u+1)t^kx^l}{u!(l-1-u)!}
 =\sum_{n=3}^{\infty}\sum_{d=0}^{n-1}\sum_{u=0}^{n-3}\sum_{l=1}^{n-2}\sum_{k=0}^{l-1}
 \sum_{j=u+2}^{n-1}\binom{n-l-2}{j-2-u}y^j \\
  \cdot \frac{A(n-l-2, d-k-1)x^{n-l-2}t^{d-k-1}}{(n-2-l)!} \cdot\frac{U(l,k,u+1)t^kx^l}{u!(l-1-u)!}.
\end{multline*}
For fixed nonnegative integer $n,l$ and $u$ such that $n-2\geq l\geq u+1$ (then $n-u-3 \geq n-l-2$), we have
\[
\sum_{j=2}^{n-1}\binom{n-l-2}{j-2-u}y^j=y^{u+2}\sum_{i=-u}^{n-u-3}
\binom{n-l-2}{i}y^{i}=(1+y)^{n-l-2}y^{u+2}.
\]
Thus
\begin{align*}
 E(t,x,y)
 &=tx^2\sum_{n=3}^{\infty}\sum_{d=0}^{n-1}\sum_{u=0}^{n-3}\sum_{l=1}^{n-2}\sum_{k=0}^{l-1}
 \sum_{j=u+2}^{n-1}\binom{n-l-2}{j-2-u}y^j \cdot
 \frac{A(n-l-2, d-k-1)x^{n-l-2}t^{d-k-1}}{(n-2-l)!} \\
 &\quad \cdot\frac{U(l,k,u+1)t^kx^l}{u!(l-1-u)!}+tx^2yA(t,x,y)-t^2x^2yA(t,x,y).\\
 &=tx^2\sum_{n=3}^{\infty}\sum_{d=0}^{n-1}\sum_{u=0}^{n-3}\sum_{l=1}^{n-2}\sum_{k=0}^{l-1}
  (1+y)^{n-l-2}y^{u+2}\cdot \frac{A(n-l-2, d-k-1)x^{n-l-2}t^{d-k-1}}{(n-2-l)!} \\
 &\quad \cdot\frac{U(l,k,u+1)t^kx^l}{u!(l-1-u)!}+tx^2yA(t,x,y)-t^2x^2yA(t,x,y)\\
 &= tx^2y \sum_{n=3}^{\infty}\sum_{d=0}^{n-1}\sum_{u=0}^{n-3}\sum_{l=1}^{n-2}\sum_{k=0}^{l-1}  \dfrac{A(n-l-2, d-k-1)[x(1+y)]^{n-l-2}t^{d-k-1}}{(n-2-l)!} \\
 &\cdot\frac{U(l,k,u+1)t^kx^ly^{u+1}}{u!(l-1-u)!}+t(1-t)x^2yA(t,x,y)\\
 &=tx^2y \sum_{n,d}\frac{A(n, d)[x(1+y)]^{n}t^{d}}{n!} \cdot\sum_{l,k}\sum_{u=1}^{k}\frac{U(l,k,u)t^kx^ly^{u}}{(u-1)!(l-u)!} +t(1-t)x^2yA(t,x,y) \\
 &=tx^2y[A(t,x+xy)+1]\hat{U}(t,x,y)+t(1-t)x^2yA(t,x,y).
\end{align*}
 	Thus completing the proof.
 \end{proof}

\subsection{The generating function of $\{b(n,d,j)\}_{n,d,j}$}

In this subsection we will give a relation between $B(t,x,y),B(t,x)$ and $E(t,x,y)$
by the relations between their coefficients. We give a bijection proof adopting the idea of the reversal-concatenation map in~\cite{WZ20a}. The concept of lowest point in ~\cite{WZ20a} play a important role in the proof, which is defined to be the position $1\leq k\leq n$ of a permutation $\pi=\pi_1\pi_2\dotsm\pi_n$ satisfying
\[
h(\pi_1\pi_2\dotsm\pi_k)=min\{h(\pi_1\pi_2\dotsm\pi_i)|1\leq\,i\leq\,n\}.
\]
From the definition, it is easy to see that if $l$ is the minimal lowest point of
a permutation $\pi=\pi_1\pi_2\cdots\pi_n$, then $\pi_{l-1}\pi_{l-2}\cdots\pi_{1}$ and
$\pi_{l}\pi_{l+1}\cdots\pi_{n}$ are both ballot permutations (empty permutation is also ballot permutation). For example, the first lowest point of permutation 143265 is 4,
then 341 and 265 are ballot permutations. Similarly, if $l$ is the maximal lowest point of
a permutation $\pi=\pi_1\pi_2\cdots\pi_n$, then $\pi_{l}\pi_{l-1}\cdots\pi_{1}$ and
$\pi_{l+1}\pi_{l+2}\cdots\pi_{n}$ are ballot permutations.
\begin{thm}\label{recendj}
For all integers $n,d$ and $j$ such that $n\geq3, 0\leq d\leq n-1$
and $2\leq j\leq n-1$, we have
\begin{multline}\label{E+E} E(n,d,j)+E(n,d-1,j)=\sum_{l,k,u}\binom{j-2}{u}\binom{n-j-1}{n-l-3-u}b(l,k)b(n-l,\,d-l+k,\,u+2)\\
+\sum_{l,k,u}\binom{j-2}{u}\binom{n-j-1}{l-3-u}b(n-l,d-l+k)b(l,\,k,\,u+2).
\end{multline}
\end{thm}
\begin{proof}
Define
\[
\mathcal{E}(n,d,j)=\{\pi\in\mathcal{S}_n:\des(\pi)=d ~\text{and $\pi$
has $1nj$ or $jn1$ as a factor}\},
\]
\[
\mathcal{B}(n,d)=\{\pi\in\mathcal{S}_n:\des(\pi)=d ~\text{and $\pi$
is a ballot permutation}\},
\]
\[
\mathcal{B}_j(n,d)=\{\pi\in\mathcal{B}(n,d): \text{$\pi$
has $1nj$ or $jn1$ as a factor}\},
\]
\begin{multline*}
\mathcal{B}^1(n,d,j)
=\{(\rho,\tau)\colon\text{$\exists$ $l,k,u\geq0$ such that} \\ \rho\in\mathcal{B}(l,k),
\tau\in\mathcal{B}_{u+2}(n-l,d-l+k), ~\text{and}~ \rho\tau\in \mathcal S_{n}\},~\text{and}
\end{multline*}
\begin{multline*}
\mathcal{B}^2(n,d,j)
=\{(\rho,\tau)\colon\text{$\exists$ $l,k,u\geq0$ such that} \\ \rho\in\mathcal{B}_{u+2}(l,k),\tau\in\mathcal{B}(n-l,d-l+k),
 ~\text{and}~ \rho\tau\in \mathcal S_{n}\}.
\end{multline*}
Now we give a bijection between $\mathcal{E}(n,d,j)\cup \mathcal{E}(n,d-1,j)$ and
$\mathcal{B}^1(n,d,j)\cup\mathcal{B}^2(n,d,j)$.

For any permutation $\pi\in\mathcal{E}(n,d,j)$, assume that $l+1$ is the minimal lowest point. Let  $\rho=\pi_l\pi_{l-1}\dotsm\pi_1$ and $\pi(2)=\pi_{l+1}\pi_{l+2}\dotsm\pi_n$. Assume that $\des(\rho)=l-k-1$. Then $\pi_l<\pi_{l+1}$, and $\rho$ or $\tau$ has $1nj$ or $jn1$ as a factor:
\begin{itemize}
	\item
	if $\rho$ has the factor $1nj$ or $jn1$. Assume that there are $u+1$ numbers less than $j$ in $\rho$, then $(\rho,\tau)\in\mathcal{B}^2(n,d,j)$.
    \item
    if $\tau$ has the factor $1nj$ or $jn1$. Assume that there are $u+1$ numbers less than $j$ in $\tau$, then $(\rho,\tau)\in\mathcal{B}^1(n,d,j)$.
\end{itemize}

Similarly, for any permutation $\pi\in\mathcal{E}(n,d-1,j)$, assume that $l$ is the maximal lowest point. Let $\rho=\pi_l\pi_{l-1}\dotsm\pi_1$ and $\pi(2)=\pi_{l+1}\pi_{l+2}\dotsm\pi_n$. Assume that $\des(\rho)=l-k-1$. Then $(\rho,\tau)\in\mathcal{B}^1(n,d,j)\bigcup\mathcal{B}^2(n,d,j)$.

It is not difficult to check that the map
\begin{align*}
\phi\colon \mathcal{E}(n,d,j)\cup \mathcal{E}(n,d-1,j) &\to
\mathcal{B}^1(n,d,j)\bigcup\mathcal{B}^2(n,d,j)\\
\pi&\mapsto(\rho,\tau)
\end{align*}
is a bijection. Thus completing the proof.
\end{proof}

With \cref{recendj}, we obtain the following relation between $B(t,x,y),B(t,x)$ and $E(t,x,y)$.
\begin{thm}\label{D1} We have that
\begin{align} \label{Btxy}
B(t,x,y)B\brk2{\frac{1}{t},tx(1+y)}+B\brk2{\frac{1}{t},tx,y}B(t,x+xy)=(1+t)E(t,x,y).
\end{align}
\end{thm}
\begin{proof}
Multiplying each term in \eqref{E+E} by $\dfrac{t^{d}x^{n}y^{j}}{(n-j-1)!(j-2)!}$,
we have
\begin{align}\label{222}
&\quad\frac{E(n,d,j)t^{d}x^{n}y^{j}}{(n-j-1)!(j-2)!}+t\cdot\frac{E(n,d-1,j)t^{d-1}x^{n}y^{j}}{(n-j-1)!(j-2)!}\\
&=\sum_{l,k,u}\dfrac{b(l,k)t^{-k}(tx)^ly^{j-u-2}}{l!}\cdot
\frac{b_{n-l,\,d-l+k}(1,u+2)t^{d-l+k}x^{n-l}y^{u+2}}{(n-l-3-u)!u!}\binom{l}{j-2-u}\\
&\quad+\sum_{l,k,u}\frac{b(n-l,d-l+k)t^{d-l+k}x^{n-l}y^{j-2-u}}{(n-l)!}
\cdot\frac{b_{l,k}(1,u+2)t^{-k}(xy)^{l}z^{u+2}}{(l-3-u)!(n-l!)}\binom{n-l}{j-2-u}.
\end{align}
Summing over \eqref{222} for all $n,d$ and $j$ such that $n\geq3, 0\leq d\leq n-1$
and $2\leq j\leq n-1$, the desired generating function can be obtained by using
techniques in generatingfunctionology as that is used in the proof of \cref{D2}.
\end{proof}

\subsection{The generating function of $\{p_{n,d}(1,j)\}_{n,d,j}$}
To calculate $P(t,x,y)$, we first prove the following formula for $p_{n,d}(1,j)$.
\begin{thm}
For $n\geq3$, $2\leq j \leq n-1$ and $0\leq d\leq n-1$, We have
\begin{multline} \label{relofpnd}
p_{n,d}(1,j)=\sum_{m_1+m_2+\lambda_1 n_1+\dotsm+\lambda_s n_s=n-3,d_0\leq \frac{m_1+m_2+2}{2},\atop d_0+\lambda_1 d_1+\dotsm+\lambda_s d_s=d,  n_i\,\text{is odd},\,m_1+m_2\,\text{is even}} \binom{j-2}{m_1}\binom{n-j-1}{m_2}
\dfrac{(n-3-m_1-m_2)!}{\Pi_{i=1}^{s}[(n_i!)^{\lambda_i}\lambda_i!]}\\
\cdot\prod_{i=1}^{s}[l(n_i,\,d_i)]^{\lambda_i} U(m_1+m_2+1,\,d_0-1,\,m_1+1).	
\end{multline}
\end{thm}
\begin{proof}
It is easy the see that, $p_{n,d}(1,j)$ is the ways of decomposition of $[n]$ satisfying
the following conditions:
\begin{itemize}
\item (i) $s,m_1,m_2,d_0,d_i,n_i,\lambda_i(1\leq i\leq s)$ are nonnegative integers such that $d_0\leq \frac{m_1+m_2+2}{2}$, $n_i$ is odd ($1\leq i\leq s$), $m_1+m_2$ is even and
    \[
    (m_1+m_2,d_0)+\sum_{i=1}^{s}\lambda_i(n_i,d_i)=(n-3,d).
    \]
\item (ii) $[n]$ is divided into $1+\sum_{i=1}^{s}\lambda_i$ odd order cycles, such that there are $\lambda_i$ cycles of length $n_i$ with $M(\cdot)=d_i$ ($1\leq i\leq s$) and the remaining cycle (denoted by $c$) has length $m_1+m_2+3$ with $M(c)=d_0$.
\item (iii) $c$ has $1nj$ as a cyclic factor and $m_1$ numbers belonging to $\{2,3,\dotsm,j-1\}$.
\end{itemize}

For odd integer $n$ and integer $0\leq d\leq \lfloor \frac{n-1}{2}\rfloor$, let $l(n,d)$ denote the number of cyclic permutations over $[n]$ of length $n$ with $M(\cdot)=d$.
When $n>1$, for any such cyclic permutation $(nc_1c_2\dotsm c_{n-1})$, the permutation $c_1c_2\dotsm c_{n-1}$ has length $n-1$ and $d-1$ or $n-d-1$ descents. Since $n$ is odd,
$d-1\neq n-d-1$. Noting that $A(n-1,d-1)=A(n-1,n-1-1-(d-1))=A(n-1,n-d-1)$, we have
\begin{equation*}
l(n,d)=\left\{
\begin{aligned}
&2A(n-1,\,d-1), \,\,\, &n>1,\,\,\, d\leq\frac{n-1}{2}, \\
&1,&n=1, \,\,\,d=0.
\end{aligned}
\right.
\end{equation*}
According to \cref{Wang and Zhao}, we have
\[
B(t,x)=\exp\brk3{x+2 \sum_{k\ge 1}\sum_{d\le k-1} A(2k,d)t^{d+1}\frac{x^{2k+1}}{(2k+1)!}}
= \exp\brk3{\sum_{n ~\text{is odd},d}l(n,d)\frac{t^{d}x^{n}}{n!}}.
\]

Similarly, assume that a odd order cyclic permutation $c=(1nj\pi_1 \pi_2\,\cdots\,\pi_{n-3})$ over $[n]$ with $M(\pi)=d$ has $m$ numbers belonging to $\{2,3,\dotsm,j-1\}$. Then the permutation
$j\pi_1 \pi_2\,\cdots\,\pi_{n-3}$ is over $\{2,3,\dotsm,n-1\}$ which have $d-2$ or $n-d-2$ descents. Then the number of such cyclic $c$ is
\[
A(n-2,\,d-2,\,j-1)+A(n-2,\,n-d-2,\,j-1)=U(n-2,d-1,j-1).
\]

Thus for fixed $s,m_1,m_2,d_0,d_i,n_i,\lambda_i(1\leq i\leq s)$
satisfying condition (i), the ways of decomposition of $[n]$ satisfying
condition (ii) and (iii) is
\begin{multline*}
\binom{j-2}{m_1}\binom{n-j-1}{m_2}
\dfrac{(n-3-m_1-m_2)!}{\Pi_{i=1}^{s}[(n_i!)^{\lambda_i}\lambda_i!]}
\prod_{i=1}^{s}[l(n_i,\,d_i)]^{\lambda_i}\\
\cdot U(m_1+m_2+1,\,d_0-1,\,m_1+1).	
\end{multline*}
Summing over the integers $s,m_1,m_2,d_0,d_i,n_i,\lambda_i(1\leq i\leq s)$
satisfying condition (i) and thus completing the proof.
\end{proof}

\begin{thm}\label{pndxyz}
We have
\begin{align}\label{ptxy}
P(t,x,y)=tx^2yB(t,x+xy)U(t,x,y).
\end{align}
\end{thm}
\begin{proof}
Multiplying each term in \eqref{relofpnd} by $\dfrac{t^{d}x^{n}y^{j}}{(j-2)!(n-j-1)!}$, we obtain
\begin{multline*}
\dfrac{p_{n,d}(1,j)t^{d}x^{n}y^{j}}{(j-2)!(n-j-1)!}
=tx^{2}y\sum_{m_1+m_2+\lambda_1 n_1+\dotsm+\lambda_s n_s=n-3,d_0\leq \frac{m_1+m_2+2}{2},\atop d_0+\lambda_1 d_1+\dotsm+\lambda_s d_s=d,  n_i\,\text{is odd},\,m_1+m_2\,\text{is even}} \prod_{i=1}^{s}[\dfrac{t^{d_i}x^{n_i}l(n_i,\,d_i)}{n_i!}]^{\lambda_i}\frac{1}{\lambda_i!}\\
\cdot\dfrac{U(m_1+m_2+1,d_0-1,m_1+1)t^{d_0-1}x^{m_1+m_2+1}y^{m_1+1}}{m_1!m_2!}
\binom{n-m_1-m_2-3}{j-m_1-2}y^{j-m_1-2},
\end{multline*}
Summing over all integers $n,d,j$ and noting that
\[
\sum_{j}\binom{n-3-m_1-m_2}{j-2-m_1}y^{j-2-m_1}=(1+z)^{n-3-m_1-m_2}
=\prod_{i=1}^{s}(1+y)^{\lambda_in_i},
\]
we have
\begin{equation}\label{jieguo}
\begin{split}
\quad&\quad \sum_{n,d,j}\dfrac{t^{d}x^{n}y^{j}p_{n,d}(1,j)}{(j-2)!(n-j-1)!}\\
&=tx^{2}y\sum_{n,d}\sum_{m_1+m_2+\lambda_1 n_1+\dotsm+\lambda_s n_s=n-3,d_0\leq \frac{m_1+m_2+2}{2}, \atop d_0+\lambda_1 d_1+\dotsm+\lambda_s d_s=d,  n_i\,\text{is odd},\,m_1+m_2\,\text{is even}} \prod_{i=1}^{s}[\dfrac{t^{d_i}x^{n_i}(1+y)^{n_i}l(n_i,\,d_i)}{n_i!}]^{\lambda_i}\frac{1}{\lambda_i!}\\
&\quad\cdot\dfrac{U(m_1+m_2+1,d_0-1,m_1+1)t^{d_0-1}x^{m_1+m_2+1}y^{m_1+1}}{m_1!m_2!} \\
&=tx^{2}yB(t,x(1+y))\sum_{m_1+m_2 ~\text{is even}}\sum_{d_0\leq \frac{m_1+m_2+2}{2}}
U(m_1+m_2+1,d_0-1,m_1+1)\\
&\quad \cdot\frac{t^{d_0-1}x^{m_1+m_2+1}y^{m_1+1}}{m_1!\,m_2!}\\
&=tx^{2}yB(t,x(1+y))\sum_{n ~\text{is odd},j}\sum_{d\leq\frac{n-1}{2}}\frac{U(n,d,j)t^{d}x^{n}y^{j}}{(j-1)!(n-j)!}\\
&=tx^{2}yB(t,x(1+y))U(t,x,y).
\end{split}
\end{equation}
Thus completing the proof.
\end{proof}

\section{Proof of \cref{thm:WZ}}
Now we are in a position to prove \cref{thm:WZ}.
\begin{proof}
The conclusion is equivalent to $B(t,x,y)=2P(t,x,y)$.
Since $B(t,x,y)$ is uniquely determined by \eqref{Btxy}, we only need to prove
\[
(1+t)E(t,x,y)= 2B\brk2{\frac{1}{t},tx(1+y)}P(t,x,y)+2B(t,x+xy)P\brk2{\frac{1}{t},tx,y}.
\]
According to \eqref{ptxy}, we just need to verify the following equation
\begin{multline} \label{final}
(1+t)E(t,x,y) = 2tx^2yB\brk2{\frac{1}{t},tx(1+y)}B\brk2{t,x+xy)}U(t,x,y)\\
+\frac{2(tx)^2y}{t}B(t,x+xy)B\brk2{\frac{1}{t},tx(1+y)}
U\brk2{\frac{1}{t},tx,y}.
\end{multline}
According to \cite[Theorem 3.7]{WZ20a}, we have
\[
B(t,x)B\brk2{\frac{1}{t},tx}=1+(1+t)A(t,x).
\]
Then
\[
B(t,x+xy)B\brk2{\frac{1}{t},tx(1+y)}=1+(1+t)A(t,x+xy),
\]
\eqref{final} is equivalent to
\[
(1+t)E(t,x,y)=2\,tx^2y\brk2{1+(1+t)A(t,x+xy)}\brk3{U(t,x,y)+U\brk2{\frac{1}{t},tx,y}}.
\]
According to \cref{utxy3}, we only need to prove
\begin{equation*}\label{final1}
(1+t)E(t,x,y)=tx^2y\brk2{1+(1+t)A(t,x+xy)}(\hat{U}(t,x,y)-\hat{U}(t,-x,y)).
\end{equation*}

Plugging the formulas of $A(t,x), \hat{U}(t,x,y)$ and $E(t,x,y)$ into \eqref{final1}, one can complete the proof.
\end{proof}	

With \cref{thm:WZ}, we can prove the following theorem.
\begin{thm}\label{RS} We have
\[
P(t,x,y,z)=\frac{2y}{1-yz}\brk2{P(t,x,z)-P\brk1{t,xyz,\frac{1}{y}}}.
\]
\end{thm}
\begin{proof}
  \begin{align*}
  P(t,x,y,z) & =\sum_{n,d}\sum_{1\leq i<j\leq n-1}\dfrac{2p_{n,d}(i,j)
  t^{d}x^{n}y^{i}z^{j}}{(j-i-1)!(n-j+i-2)!} \,(~\text{let} \,u=j-i+1)\\
  &=\sum_{n,d}\sum_{u=2}^{n-1}\sum_{i=1}^{n-u}\dfrac{2p_{n,d}(1,u)
  t^{d}x^{n}y^{i}z^{u+i-1}}{(u-2)!(n-u-1)!} \\
  &=\frac{2y}{1-yz} \sum_{n,d}\sum_{u=2}^{n-1}\dfrac{p_{n,d}(1,u)
  (t^{d}x^{n}z^{u}-t^{d}(xyz)^ny^{-u})}{(u-2)!(n-u-1)!} \\
  &=\frac{2y}{1-yz}\brk2{P(t,x,z)-P\brk1{t,xyz,\frac{1}{y}}}.
  \end{align*}
\end{proof}

\end{document}